\documentclass[12pt]{article}

\usepackage{amsfonts,dsfont}
\usepackage{amssymb}
\usepackage{epsf}
\usepackage{pstricks,pst-node,pst-text,pst-3d}
\usepackage{amsmath,stmaryrd,amsthm}
\usepackage[mathscr]{eucal}
\usepackage{tikz}

\usepackage{enumerate}
\usepackage{epsf}
\usepackage{graphicx}
\usepackage{titlesec}

\usepackage{amsthm}
\usepackage{amsmath}
\usepackage{bm}

\newtheorem{definition}{Definition}
\newtheorem{theorem}{Theorem}

\newtheorem{lemma}{Lemma}

\begin{document}

\title{\textbf{Empirical phi-divergence test statistics in the logistic regression
model}}

\author{}
\author{A. Felipe$^{1}$, P. Garc\'{\i}a-Segador$^{2}$, N. Mart\'{\i}n$^{3}$, P.
Miranda$^{1}$\thanks{Corresponding author: P. Miranda. pmiranda@mat.ucm.es. Tel: (+34) 91 394 44 19. Fax: (+34) 91 394 44 06}, and L. Pardo$^{1}$
\and $^{1}${\small Dept. of Statistics and O.R., Complutense University of Madrid, Madrid, Spain}
\and $^{2}${\small National Statistics Institute, Madrid, Spain}
\and $^{3}${\small Dept. of Financial and Actuarial Economics and Statistics, Complutense University of
Madrid}}

\date{}

\maketitle

\begin{abstract}
In this paper we apply divergence measures to empirical likelihood applied to logistic regression models. We define a family of empirical test statistics based on divergence measures, called empirical phi-divergence test statistics, extending the empirical likelihood ratio test. We study the asymptotic distribution of these empirical test statistics, showing that it is the same for all the test statistics in this family, and the same as the classical empirical likelihood ratio test. Next, we study the power function for the members in this family, showing that the empirical phi-divergence tests introduced in the paper are consistent in the Fraser sense. In order to compare the differences in behavior among the empirical phi-divergence test statistics in this new family, considered for the first time in this paper, we carry out a simulation study.

\end{abstract}

\section{Introduction}

Empirical likelihood is a tool for dealing with data in nonparametric conditions but allowing the use of results derived from parametric inference. First introduced by Owen \cite{owe88, owe90, owe91, owe03}, empirical likelihood has been applied successfully to many different statistical problems, leading to results that can be applied in general situations. Roughly speaking, empirical likelihood offers techniques that can be applied to non-parametric problems while using tools designed for parametric cases.

One of these situations appears in the logistic regression model. As it will become apparent below, the empirical likelihood theory applied to the model of logistic regression is based on maximum likelihood; thus, the parameters of the model are obtained via this technique, and so is the corresponding likelihood ratio test for testing if the parameters take a concrete value. This test statistic is given by the difference between the likelihood based on the empirical maximum likelihood estimators (EMLE) and the likelihood in accordance with the null hypothesis. It can be shown (see below) that under mild conditions, EMLE is attained when the uniform distribution is considered. In this case, the empirical likelihood ratio test statistic (ELRT) can be seen as the Kullback-Leibler divergence measure between the uniform distribution and the best distribution (again in terms of likelihood) under the null hypothesis. On the other hand, the Kullback-Leibler divergence measure is just one of the members of a wide family of divergence measures known as phi-divergence measures. Thus, it makes sense to study what happens if other divergence measures of this family are considered instead. Indeed, phi-divergence test statistics have been considered previously in the context of empirical likelihood (see \cite{bamapa15, bamapa17, femimapa18} and references therein).

In this paper we introduce a new family of empirical test statistics and show that the asymptotic distribution for any member of this new family, named empirical phi-divergence test statistics, is the same and coincides with the asymptotic distribution of the ELRT. Therefore, they have the same behavior for large sample sizes and differences can only arise for small and moderate sample sizes. In order to compare these empirical phi-divergence test statistics, a simulation study is carried out. Next, we also study the power function of the empirical phi-divergence test statistics, showing that under any hypothesis in the alternative, the asymptotic distribution of the phi-divergence test statistics follows a normal distribution. In this case the parameters depend on the measure of phi-divergence considered. However, these differences tend to vanish when the sample size grows.

To fix the notation that we will use throughout the paper, let us develop the application of empirical likelihood to logistic regression (see \cite{roc01}). Consider a random sample of $n$ independent and identically distributed $(q+1)$-random vectors $(\boldsymbol{X}_{1},Y_{1}), ..., (\boldsymbol{X}_{n},Y_{n})$ from the random vector $(\bm X, Y).$ Component $Y$ is a dichotomous variable, and it is assumed that $P(Y=1)$ depends on the explanatory vector $\bm X =(X_1, ..., X_q)$. This means that for $Y_i, i=1, ..., n$, it follows

\[
\Pr\left(  Y_{i}=1\right)  =\pi_{i}\text{ and }\Pr\left(  Y_{i}=0\right)
=1-\pi_{i},\text{ }i=1,...,n,
\]
where $\pi_i$ can be explained in terms of $X_{1i}, ..., X_{qi}$ and some unknown parameters $\beta_1, ..., \beta_q$ such that $\beta_{i}\in\left(
-\infty,\infty \right) , i=1,...,q.$ Thus, the parametric space is $\Theta = \mathbb{R}^q.$ We will use the notation $\boldsymbol{x}_{i}^{T}:=\left(  x_{i1},...,x_{iq}\right) $ for the values of $\left( X_{1i}, ..., X_{qi}\right) $ in the sample and $\boldsymbol{\beta}^{T}:=\left(  \beta_{1},...,\beta_{q}\right) $.

It is assumed that $\pi_{i}=\pi (\boldsymbol{x}_{i}^{T}\boldsymbol{\beta}),\, i=1, ..., n$, i.e. it depends on $\boldsymbol{x}_{i}^{T}\boldsymbol{\beta}$. When dealing with the logistic regression model, this dependence is given by the logit function

\[ log \left( {\pi_i\over 1-\pi_i} \right) =: logit\left( \pi_{i}\right) =\sum_{j=1}^{q}\beta_{j}x_{ij}. \]

The explanatory design matrix is
\[ \mathds{X}=\left(  \boldsymbol{X}_{1},...,\boldsymbol{X}_{n}\right)  ^{T} \]
and we assume that $rank(\mathds{X})=q.$ We aim to obtain the value of parameter $\boldsymbol{\beta}.$ For a given sample

\[ \{ \left(  \boldsymbol{x}_{i}^{T},Y_{i}\right) , i=1,...,n\} \]
of $n$ independent observations, classical inference for the unknown parameter vector $\boldsymbol{\beta}$ is usually based on the maximum likelihood estimator (MLE). It is well-known that in general, and in particular for the logistic regression model, the MLE is a BAN (Best Asymptotically Normal) estimator.

Assuming the logistic regression model holds, we can apply the Lagrange multiplier method to conclude that the MLE $\hat{\boldsymbol{ \beta}}$ will be a solution of the following estimating equations,

\begin{equation}
\sum_{i=1}^{n} \boldsymbol{g}\left( \boldsymbol{x}_{i},Y_{i},\boldsymbol{\beta}\right) =\boldsymbol{0}_q, \label{0.1}
\end{equation}
where
\[ \boldsymbol{g}\left( \boldsymbol{x}_{i},Y_{i},\boldsymbol{\beta}\right) := \boldsymbol{x}_{i} (Y_{i}-\pi(\boldsymbol{x}_{i}^{T}\boldsymbol{\beta})), \quad i=1, ..., n. \]

Following \cite{qila94}, it can be seen that this system has a unique solution provided $\boldsymbol{0}$ is inside the convex hull of the points $\boldsymbol{g}\left( \boldsymbol{x}_{i},Y_{i},\boldsymbol{\beta}\right),\, i=1, ..., n.$


Note that it could happen that $\pi_i\approx 1$ and $Y_i=0$ or $\pi_i\approx 0$ and $Y_i=1$ and these data have a big impact on the estimations. In order to reduce the impact of these unusual data, some authors \cite{biyo96} have proposed a modification of function $\boldsymbol{g}$ introducing some weights. Thus, the new function would be defined by

\[ \boldsymbol{g}\left( \boldsymbol{x}_{i},Y_{i},\bm \beta\right) := w_i\boldsymbol{x}_{i} (Y_{i}-\pi(\boldsymbol{x}_{i}^{T}\bm \beta)), \quad i=1, ..., n, \]
where $w_i$ approaches 0 if $\pi_i\approx 1$ and $Y_i=0$ or $\pi_i\approx 0$ and $Y_i=1$. In this paper we consider $w_i=1, \, \forall i$, and so we do not measure the rationality of the data.

From now on, we assume the following condition:

\begin{itemize}
\item {\bf C1}. The true parameter $\bm \beta_0$ is the only solution of

\[ E\left[ \boldsymbol{g}\left( \boldsymbol{X}, Y,\bm \beta \right) \right] =\boldsymbol{0}_q. \]

\end{itemize}



Let $p_{1}(\bm \beta),...,p_{n}(\bm \beta )$ be a set of probability weights (depending on $\bm \beta $) allocated to the data. The EMLE (see \cite{owe88}), $\hat{\bm \beta}_{E, n}$ looks for the values of these weights so that the likelihood function under the logistic regression model is maximum; that translates into solving

\begin{equation}
\mathcal{L}_{n}\left(  \boldsymbol{\beta}\right)  =\max_{\bm \beta \in \mathbb{R}^q} \prod\limits_{i=1}^{n}p_{i}(\boldsymbol{\beta}) \label{0.2}
\end{equation}
subject to the constrains

\begin{equation}
\sum\limits_{i=1}^{n}p_{i}(\bm \beta )=1 \label{0.3}%
\end{equation}
and

\begin{equation}
\sum\limits_{i=1}^{n}p_{i}(\bm \beta )\boldsymbol{g}\left(  \boldsymbol{X}_{i},Y_{i},\bm \beta \right)  =\boldsymbol{0}_q. \label{0.4}%
\end{equation}

Remark that different values of $\bm \beta$ lead to different values for $\boldsymbol{g}\left(  \boldsymbol{X}_{i},Y_{i},\bm \beta \right)$, so this justifies that the probabilities $p_{1}(\bm \beta),...,p_{n}(\bm \beta )$ depend on $\bm \beta .$

Note also that we need the condition $p_i(\bm \beta )\geq 0, \forall i=1, ..., n$. As stated in \cite{roc01, roc03}, if we apply the Lagrange multipliers method, we obtain

%
%
$$ \sum_{i=1}^n \ln p_i(\boldsymbol{\beta }) + \lambda_0(\bm \beta ) (\sum_{i=1}^n p_i(\boldsymbol{\beta }) -1) - \boldsymbol{\lambda }^T(\bm \beta ) \sum_{i=1}^n  p_{i}(\bm \beta )\boldsymbol{g}\left(  \boldsymbol{X}_{i},Y_{i},\bm \beta \right).$$

Derivating with respect to $p_i(\bm \beta)$ and equalizing to 0 leads to

$$ {1\over p_i(\boldsymbol{\beta })} + \lambda_0(\bm \beta ) - \boldsymbol{\lambda }^T(\bm \beta )\boldsymbol{g}\left(  \boldsymbol{X}_{i},Y_{i},\bm \beta \right)=0 \Leftrightarrow 1 + \lambda_0(\bm \beta ) p_i(\boldsymbol{\beta })- p_i(\boldsymbol{\beta })\boldsymbol{\lambda }^T(\bm \beta )\boldsymbol{g}\left(  \boldsymbol{X}_{i},Y_{i},\bm \beta \right)=0.$$

Now, summing up in $i$ gives

$$ n + \lambda_0(\bm \beta ) \sum_{i=1}^n p_i(\boldsymbol{\beta })+ \boldsymbol{\lambda }^T(\bm \beta ) \sum_{i=1}^np_i(\boldsymbol{\beta })\boldsymbol{g}\left(  \boldsymbol{X}_{i},Y_{i},\bm \beta \right)=0,$$
and hence, as ${\displaystyle \sum_{i=1}^n p_i(\boldsymbol{\beta })=1}$ and ${\displaystyle \sum_{i=1}^np_i(\boldsymbol{\beta })\boldsymbol{g}\left(  \boldsymbol{X}_{i},Y_{i},\bm \beta \right)=0,}$ we obtain

$$ n + \lambda_0(\bm \beta ) =0 \Leftrightarrow \lambda_0(\bm \beta )=-n.$$

But then,

$$ 1 - n p_i(\boldsymbol{\beta })- p_i(\boldsymbol{\beta })\boldsymbol{\lambda }^T(\bm \beta )\boldsymbol{g}\left(  \boldsymbol{X}_{i},Y_{i},\bm \beta \right)=0 \Leftrightarrow p_i(\boldsymbol{\beta })= \frac{1}{n}\frac{1}{1+{1\over n}\boldsymbol{\lambda }^{T}\left( \bm \beta \right)  \boldsymbol{g}\left( \boldsymbol{X}_{i},Y_{i},\bm \beta \right) }.$$

Defining $\boldsymbol{t}(\bm \beta )= {1\over n}\boldsymbol{\lambda }(\bm \beta ),$ we obtain the solution.

\begin{equation}
p_{i}\left( \bm \beta \right) =\frac{1}{n}\cdot \frac{1}{1+\boldsymbol{t}^{T}\left( \bm \beta \right)  \boldsymbol{g}\left( \boldsymbol{X}_{i},Y_{i},\bm \beta \right) },\quad i=1,...,n, \label{0.4BIS}
\end{equation}
where $\boldsymbol{t}^{T}\left( \bm \beta \right) :=\left( t_{1}\left( \bm \beta \right) ,...,t_{q}\left( \bm \beta \right) \right) $ is a $q-$dimensional vector solution of the non-linear system of $q$ equations

\begin{equation}\label{5bis} \frac{1}{n}\sum\limits_{i=1}^{n}\frac{\boldsymbol{g}\left(  \boldsymbol{X}_{i},Y_{i},\bm \beta \right)  }{1+\boldsymbol{t}^{T}\left(
\bm \beta \right)  \boldsymbol{g}\left(  \boldsymbol{X}_{i},Y_{i},\bm \beta \right)  }=\boldsymbol{0}_q. \end{equation}


Besides, we also need $0\leq p_i \left( \bm \beta \right) \leq 1;$ hence, we need that

$$1+\boldsymbol{t}^{T}\left(
\bm \beta \right)  \boldsymbol{g}\left(  \boldsymbol{X}_{i},Y_{i},\bm \beta \right) \geq {1\over n},\, \forall i=1, ..., n.$$

We assume the following conditions:

\begin{itemize}
\item {\bf C2}: The null vector is contained in the convex hull of $\boldsymbol{g}\left(  \boldsymbol{X}_{i},Y_{i},\bm \beta \right) ,\, i=1, ..., n.$
\item {\bf C3}: The matrix ${\displaystyle \sum_{i=1}^n \boldsymbol{g}\left(  \boldsymbol{X}_{i},Y_{i},\bm \beta \right) \boldsymbol{g}\left(  \boldsymbol{X}_{i},Y_{i},\bm \beta \right)^T}$ is positive definite.
\end{itemize}

Then, we have the following result, whose proof can be seen in \cite{qila94}.

\begin{lemma}\label{lemma1}
Assuming {\bf C2} and {\bf C3}, there is a unique solution of the system defined by Eq. \eqref{5bis}, and the solution defined through Eq. \eqref{0.4BIS} is the only solution of the problem given by Eqs. \eqref{0.2}, \eqref{0.3} and \eqref{0.4}. Moreover, $\boldsymbol{t}^{T}\left(
\bm \beta \right) $ is continuous on $\bm \beta .$
\end{lemma}



Thus, the kernel of the empirical log-likelihood function is
\begin{equation}
\ell_{E,n}\left( \boldsymbol{\beta}\right) =-\sum\limits_{i=1}^{n} \log \left(  1+\boldsymbol{t}^{T}\left( \bm \beta \right) \boldsymbol{g}\left( \boldsymbol{X}_{i},Y_{i},\bm \beta \right) \right) . \label{0.5}
\end{equation}

Suppose that the logistic regression model holds and consider the problem of testing
\begin{equation}
H_{0}:\bm \beta =\bm \beta_{0}\text{ versus }H_{1}:\bm \beta \neq \bm \beta_{0}. \label{0.6}
\end{equation}

Based on Qin and Lawless \cite{qila94}, the ELRT for testing (\ref{0.6}) is given by
\[ \emph{L}_{E,n}(\widehat{\bm \beta}_{E,n},\bm \beta_{0})=2\ell_{E,n}(\widehat{\bm \beta}_{E,n})-2\ell_{E,n}
(\bm \beta_{0}).\]


Then, we are comparing the likelihood of the sample for $\bm \beta_0$ and $\widehat{\bm \beta}_{E,n}$. It is obvious that the maximum likelihood is attained when $\boldsymbol{t}(\bm \beta )= \boldsymbol{0};$ if this is so, then $p_i(\bm \beta ) ={1\over n}.$ For this to happen, it is necessary that there is a solution for

$$ \sum\limits_{i=1}^{n} \boldsymbol{g}\left(  \boldsymbol{X}_{i},Y_{i},\bm \beta \right)  =\boldsymbol{0}_q, $$ for some $\bm \beta $. It can be shown that as the number of parameters and the dimension of $\boldsymbol{g}$ are the same, then such system has solution with probability approaching to 1 in large samples (see for instance \cite{qila94}). Thus, $\boldsymbol{t}(\widehat{\bm \beta}_{E,n})=\boldsymbol{0},$ and hence,


\begin{equation}\label{L}
\boldsymbol{p}(\widehat{\bm \beta}_{E,n})=\boldsymbol{u}=\left(  \tfrac{1}{n},....,\tfrac{1}{n}\right)  ^{T}.
\end{equation}

Therefore,

\[ \ell_{E,n}(\widehat{\bm \beta}_{E,n})=0\]
and thus,
\begin{equation}\label{0.7}
\emph{L}_{E,n}(\widehat{\bm \beta}_{E,n},\bm \beta_{0})=-2\ell_{E,n}(\bm \beta_{0}).
\end{equation}

Let us consider the following conditions:

\begin{itemize}
\item {\bf C4:} The matrix $E\left[ \boldsymbol{g}\left(  \boldsymbol{X}_{i},Y_{i},\bm \beta \right) \boldsymbol{g}^T\left(  \boldsymbol{X}_{i},Y_{i},\bm \beta \right) \right] $ is positive definite.
\item {\bf C5:} ${\partial \boldsymbol{g}\left(  \boldsymbol{X}_{i},Y_{i},\bm \beta \right) \over \partial \bm \beta }$ is continuous in a neighborhood of the true parameter $\bm \beta_0.$
\item {\bf C6:} $\| {\partial \boldsymbol{g}\left(  \boldsymbol{X}_{i},Y_{i},\bm \beta \right) \over \partial \bm \beta } \| $ and $\| \boldsymbol{g}\left(  \boldsymbol{X}_{i},Y_{i},\bm \beta \right)^3 \| $ are bounded by some integrable function $G_1\left[ \left(  \boldsymbol{X}_{i},Y_{i} \right) \right] $ in a neighborhood of the true parameter $\bm \beta_0$.
\item {\bf C7:} The rank of $E\left[ {\partial \boldsymbol{g}\left(  \boldsymbol{X}_{i},Y_{i},\bm \beta \right) \over \partial \bm \beta }\right] $ is $q$.
\item {\bf C8:} ${ \partial^2 \boldsymbol{g}\left(  \boldsymbol{X}_{i},Y_{i},\bm \beta \right) \over \partial \bm \beta \partial \bm \beta^T } $ is continuous in a neighborhood of the true parameter $\bm \beta_0.$
\item {\bf C9:} $\| {\partial^2 \boldsymbol{g}\left(  \boldsymbol{X}_{i},Y_{i},\bm \beta \right) \over \partial \bm \beta \partial \bm \beta^T } \| $ is bounded by some integrable function $G_2[\left(  \boldsymbol{X}_{i},Y_{i} \right) ] $ in a neighborhood of the true parameter $\bm \beta_0$.
\end{itemize}

Now, the following result was established in \cite{qila94}.

\begin{theorem}
Under the conditions {\bf C1}-{\bf C9}, the ELRT for testing $H_0: \bm \beta = \bm \beta_0$ versus $H_1: \beta \ne \beta_0$ given by $L_{E,n}(\widehat{\bm \beta}_{E,n},\bm \beta_{0})$ defined in Eq. \eqref{0.7} satisfies

$$ L_{E,n}(\widehat{\bm \beta}_{E,n},\bm \beta_{0}) \underset{n\rightarrow\infty}{\overset{\mathcal{L}}{\rightarrow}} \chi^2_q$$ if $H_0$ is true.
\end{theorem}

As a consequence, an approximate confidence region of the nominal level
$1-\alpha$ for $\bm \beta_{0}$ is given by

\[ C_{1-\alpha}=\{ \bm \beta \text{ }\boldsymbol{/}\text{ }\emph{L}_{E,n}(\widehat{ \bm \beta}_{E,n},\bm \beta )<\chi_{q,1-\alpha}^{2}\} .\]

If inferences on a subset of $k$ coefficients are of interest, with $k<q$, the remaining $q-k$ coefficients can be profiled out, treating them as nuisance parameters, and then the asymptotic distribution is a chi-square distribution with $k$ degrees of freedom \cite{qila94}.

The rest of the paper is as follows: Section 2 is devoted to introduce and study a family of empirical phi-divergence test statistics based on phi-divergence measures, as well as to obtain its asymptotic distribution and an approximation to the power function. In Section 3, we present a simulation study to compare the behavior of the different members of the family of empirical phi-divergence test statistics introduced in Section 2. We finish with some conclusions.

\section{Empirical phi-divergence test statistics}

Consider two probability distributions $\boldsymbol{p}^T=(p_1, ..., p_n)$ and $\boldsymbol{q}^T=(q_1, ..., q_n)$ defined on the same finite referential set. We define the {\bf Kullback-Leibler divergence} between $\boldsymbol{p}$ and $\boldsymbol{q}$ as

$$ d_{K-L}(\boldsymbol{p}, \boldsymbol{q})= \sum_{i=1}^n p_i \log {p_i\over q_i},$$ where $0\log 0=0$ by convention; it measures the extent $\boldsymbol{q}$ separates from $\boldsymbol{p}$. Consider the $n$-dimensional probability vector $\boldsymbol{u}$ defined in (\ref{L}), as well as the $n$-dimensional probability vector

\[ \boldsymbol{p}\left(  \bm \beta \right)  =\left( p_{1}\left( \bm \beta \right)  ,...,p_{n}\left(  \bm \beta \right)
\right)^{T}\]
with $p_{i}\left( \bm \beta \right) , i=1,...,n$, defined in (\ref{0.4BIS}). Then, the Kullback-Leibler divergence between $\boldsymbol{u}$ and $\boldsymbol{p}\left( \bm \beta \right) $ is
given by

\begin{align*}
d_{K-L}(\boldsymbol{u,p}\left( \bm \beta \right) ) & =\sum_{i=1}^{n} u_{i}\log\frac{u_{i}}{p_{i}\left( \bm \beta \right) }\\
& =\sum_{i=1}^{n}\frac{1}{n}\log\frac{\frac{1}{n}}{\frac{1}{n}\frac{1}{1+\boldsymbol{t}^{T}\left( \bm \beta \right) g\left( \boldsymbol{X}_{i},Y_{i},\bm \beta \right) }}\\
& =\frac{1}{n}\sum_{i=1}^{n}\log\left( 1+\boldsymbol{t}^{T}\left( \bm \beta \right) g\left( \boldsymbol{X}_{i},Y_{i},\bm \beta \right) \right) .
\end{align*}
Therefore,
\begin{equation}
\emph{L}_{E,n}(\widehat{\bm \beta}_{E,n},\bm \beta_{0})=2n d_{K-L}(\boldsymbol{u},\boldsymbol{p}(\bm \beta_{0})).
\label{0.8}
\end{equation}
On the other hand, it must be taken into account that the Kullback-Leibler divergence is a member of a broad family of measures of divergence known as {\bf phi-divergence measures}. For introducing phi-divergence measures, let us denote by $\Phi^{\ast}$ the class of all twice differentiable convex functions $\phi:\mathbb{R}^{+}\rightarrow\mathbb{R}$ such that $\phi\left(
1\right)  =0$, $\phi^{\prime\prime}\left( 1\right) >0$, $0\phi\left( 0/0\right) =0$, and $0\phi\left( p/0\right) =p\lim_{u\rightarrow\infty}\frac{\phi\left(  u\right)  }{u}$. For $\phi \in \Phi^{\ast},$ we define the $\phi $-divergence measure between two probability vectors $\boldsymbol{p}, \boldsymbol{q}$ by

\begin{equation}
d_{\phi }(\boldsymbol{p}, \boldsymbol{q})= \sum_{i=1}^{n} q_i \phi \left( {p_i\over q_i}\right).
\end{equation}

In this context, when $\phi\left( x\right) =x\log x-x+1$, the Kullback-Leibler divergence measure is recovered. For more details about $\phi $-divergence measures see \cite{crpa02, par06}. Note that for $\phi\left( x\right) =x\log x-x+1$, it is $\phi ''(1)=1.$ Thus,

\begin{equation}
\emph{L}_{E,n}(\widehat{\bm \beta}_{E,n},\bm \beta_{0})=2n d_{K-L}(\boldsymbol{u},\boldsymbol{p}(\bm \beta_{0}))=\frac{2n}{\phi^{\prime\prime}(1)}d_{\phi}(\boldsymbol{u},\boldsymbol{p}(\bm \beta_{0})).
\end{equation}

If instead of $\phi\left( x\right) =x\log x-x+1$ we consider another functions $\phi$ belonging to $\Phi^{\ast},$ we obtain a new family of empirical test statistics for testing (\ref{0.6}): the empirical phi-divergence test statistics.

\begin{definition}
Assuming {\bf C2} and {\bf C3}, the family of empirical phi-divergence test statistics for testing (\ref{0.6}) is defined by
\begin{equation}
T_{n}^{\phi}(\bm \beta_{0})= \frac{2n}{\phi^{\prime\prime}(1)}d_{\phi}(\boldsymbol{u},\boldsymbol{p}(\bm \beta_{0})) = \frac{2}{\phi^{\prime\prime}(1)}\sum_{i=1}^n {1\over 1+\boldsymbol{t}^T(\bm \beta_0)\boldsymbol{g}\left(  \boldsymbol{X}_{i},Y_{i},\bm \beta_{0} \right)} \phi \left( 1+\boldsymbol{t}^T(\bm \beta_0)\boldsymbol{g}\left(  \boldsymbol{X}_{i},Y_{i},\bm \beta_{0} \right) \right) ,\label{0.9}
\end{equation}
with $\phi \in \Phi^{\ast}$.
\end{definition}

Remark that {\bf C2} and {\bf C3} are needed in order to apply Lemma \ref{lemma1} and hence ensure that $\boldsymbol{t}(\bm \beta_0)$ is well-defined.

For every $\phi\in \Phi^{\ast}$, the
function
\[ \Psi\left( x\right) =\phi\left( x\right) -\phi^{\prime}\left( 1\right) \left( x-1\right) \]
also belongs to $\Phi^{\ast}.$ Moreover, it can be easily seen that
\[ T_{n}^{\phi}(\boldsymbol{\beta}_{0})= T_{n}^{\psi}(\boldsymbol{\beta}_{0})\]
and $\psi$ has the additional property that $\psi^{\prime}\left( 1\right) =0.$ Since the two divergence measures are equivalent, we can consider the set $\Phi^{\ast}$ to be equivalent to the set $\Phi=\Phi^{\ast}\cap\left\{ \phi:\phi^{\prime}\left( 1\right) =0\right\} $. In what follows, we shall
assume that $\phi\in\Phi.$

\subsection{Asymptotic distribution}

In the following theorem we present the asymptotic distribution of
$T_{n}^{\phi}(\bm \beta_{0}), \phi \in \Phi $ when the null hypothesis in (\ref{0.6}) holds.

\begin{theorem}\label{theo1}
Under $H_{0}$ given in (\ref{0.6}), and conditions {\bf C1}-{\bf C9}, the asymptotic distribution of the
empirical phi-divergence test statistics $T_{n}^{\phi}(\bm \beta_{0})$ satisfies
\[ T_{n}^{\phi}(\bm \beta_{0})\underset{n\rightarrow\infty}{\overset{\mathcal{L}}{\rightarrow}}\chi
_{q}^{2}, \forall \phi \in \Phi . \]
\end{theorem}

\begin{proof}
A Taylor expansion of second order of


$$ f(x):={1\over 1+x} \phi (1+x)$$ around $x=0$ gives, taking into account that $\phi (1)=\phi'(1)=0$,

$$ f(x)={1\over 2} \phi''(1)x^2 + o(x^2),$$

Therefore, denoting $x:=\boldsymbol{t}^T(\bm \beta_0) \boldsymbol{g}\left(  \boldsymbol{X}_{i},Y_{i},\bm \beta_{0} \right) ,$ we obtain that

\begin{eqnarray*}
{1\over 1+ \boldsymbol{t}^T(\bm \beta_0) \boldsymbol{g}\left(  \boldsymbol{X}_{i},Y_{i},\bm \beta_{0} \right)} \phi (1+\boldsymbol{t}^T(\bm \beta_0) \boldsymbol{g}\left(  \boldsymbol{X}_{i},Y_{i},\bm \beta_{0} \right)) & = &  {1\over 2} \phi''(1) \boldsymbol{t}^T(\bm \beta_0) \boldsymbol{g}\left(  \boldsymbol{X}_{i},Y_{i},\bm \beta_{0} \right) \boldsymbol{g}^T\left(  \boldsymbol{X}_{i},Y_{i},\bm \beta_{0} \right) \boldsymbol{t}(\bm \beta_0) \\
& & + o( \boldsymbol{t}^T(\bm \beta_0) \boldsymbol{g}\left(  \boldsymbol{X}_{i},Y_{i},\bm \beta_{0} \right) \boldsymbol{g}^T\left(  \boldsymbol{X}_{i},Y_{i},\bm \beta_{0} \right) \boldsymbol{t}(\bm \beta_0) ).
\end{eqnarray*}

Hence,

\begin{equation}\label{A}
{2n d_{\phi }(\boldsymbol{u},\boldsymbol{p}(\bm \beta_{0}))\over \phi^{''} (1)} =  \sum_{i=1}^n \boldsymbol{t}^T (\bm \beta_{0}) \boldsymbol{g}\left(  \boldsymbol{X}_{i},Y_{i},\bm \beta_{0} \right) \boldsymbol{g}^T\left(  \boldsymbol{X}_{i},Y_{i},\bm \beta_{0} \right) \boldsymbol{t} (\bm \beta_{0}) + o\left( \sum_{i=1}^n \boldsymbol{t}^T (\bm \beta_{0}) \boldsymbol{g}\left(  \boldsymbol{X}_{i},Y_{i},\bm \beta_{0} \right) \boldsymbol{g}^T\left(  \boldsymbol{X}_{i},Y_{i},\bm \beta_{0} \right) \boldsymbol{t} (\bm \beta_{0})\right) .
\end{equation}

%

Let us write $\boldsymbol{t}(\bm \beta_0) = \rho \boldsymbol{\theta }(\bm \beta_0)$, where $\rho \geq 0$ and $||\boldsymbol{\theta }(\bm \beta_0)||=1.$ Now,

\begin{eqnarray*}
0 & = & \left\| \frac{1}{n}\sum\limits_{i=1}^{n}\frac{\boldsymbol{g}\left(  \boldsymbol{X}_{i},Y_{i},\bm \beta_0 \right) }{1+\rho \boldsymbol{\theta }^{T}\left( \bm \beta_0 \right)  \boldsymbol{g}\left( \boldsymbol{X}_{i},Y_{i},\bm \beta_0 \right) } \right\| \\ 
& \geq & \left|  \boldsymbol{\theta }^{T}\left( \bm \beta_0 \right) \frac{1}{n}\sum\limits_{i=1}^{n}\frac{\boldsymbol{g}\left( \boldsymbol{X}_{i},Y_{i},\bm \beta_0 \right) }{1+\rho \boldsymbol{\theta }^{T}\left( \bm \beta_0 \right) \boldsymbol{g}\left( \boldsymbol{X}_{i},Y_{i},\bm \beta_0 \right) } \right| \\ 
& = & \frac{1}{n} \left| \boldsymbol{\theta }^{T}\left( \bm \beta_0 \right) \left( \sum\limits_{i=1}^{n} \boldsymbol{g}\left( \boldsymbol{X}_{i},Y_{i},\bm \beta_0 \right)  - \rho \sum\limits_{i=1}^{n}\frac{\boldsymbol{g}\left( \boldsymbol{X}_{i},Y_{i},\bm \beta_0 \right) \boldsymbol{\theta }^{T}\left( \bm \beta_0 \right) \boldsymbol{g}\left( \boldsymbol{X}_{i},Y_{i},\bm \beta_0 \right) }{1+\rho \boldsymbol{\theta }^{T}\left( \bm \beta_0 \right) \boldsymbol{g}\left( \boldsymbol{X}_{i},Y_{i},\bm \beta_0 \right) } \right) \right| \\ 
& \geq & {\rho \over n} \boldsymbol{\theta }^{T}\left( \bm \beta_0 \right) \sum\limits_{i=1}^{n}\frac{\boldsymbol{g}\left( \boldsymbol{X}_{i},Y_{i},\bm \beta_0 \right) \boldsymbol{g}^T\left( \boldsymbol{X}_{i},Y_{i},\bm \beta_0 \right) }{1+\rho \boldsymbol{\theta }^{T}\left( \bm \beta_0 \right) \boldsymbol{g}\left( \boldsymbol{X}_{i},Y_{i},\bm \beta_0 \right) } \boldsymbol{\theta }\left( \bm \beta_0 \right) -{1\over n} \left| \boldsymbol{\theta }^{T}\left( \bm \beta_0 \right) \sum\limits_{i=1}^{n} \boldsymbol{g}\left( \boldsymbol{X}_{i},Y_{i},\bm \beta_0 \right) \right| \\ 
& \geq & {\rho \boldsymbol{\theta }^{T}\left( \bm \beta_0 \right) \boldsymbol{S} \boldsymbol{\theta }\left( \bm \beta_0 \right) \over 1+\rho Z_n} - {1\over n} \left| \boldsymbol{\theta }^{T}\left( \bm \beta_0 \right) \sum_{i=1}^n \boldsymbol{g}\left(  \boldsymbol{X}_{i},Y_{i},\bm \beta_0 \right) \right| ,
\end{eqnarray*}
where $Z_n= \max_{1\leq i\leq n} || \boldsymbol{g}\left(  \boldsymbol{X}_{i},Y_{i},\bm \beta_0 \right) ||$ and

\begin{equation}\label{defS} \boldsymbol{S}={1\over n} \sum_{i=1}^n \boldsymbol{g}\left( \boldsymbol{X}_{i},Y_{i},\bm \beta \right) \boldsymbol{g}^T\left( \boldsymbol{X}_{i},Y_{i},\bm \beta \right) .\end{equation}

Now, $\boldsymbol{\theta }^{T}\left( \bm \beta_0 \right) \boldsymbol{S} \boldsymbol{\theta }\geq \sigma_p + o_p(1)$, where $\sigma_p$ is the smallest (but positive by {\bf C4}) eigenvalue of $\bm \Sigma $, the covariance matrix of $ \boldsymbol{g}\left( \boldsymbol{X}_{i},Y_{i},\bm \beta_0 \right) .$ 

On the other hand, the second term is $O_p(n^{-1/2})$ by the Central Limit Theorem and {\bf C1}. Consequently, 

$$ {\rho \over 1+ \rho Z_n} = O_p(n^{-1/2}),$$ so that $\rho = || \boldsymbol{t} \left( \bm \beta_0 \right) || = O_p(n^{-1/2}).$ Now, it can be seen in \cite{owe90} that given $Y_i\geq 0$ i.i.d. and $E[Y_i^2]<\infty ,$ then ${\displaystyle \max_{1\leq i\leq n} Y_i = o(n^{1/2}).}$ Defining $Y_i:= |\boldsymbol{g}\left( \boldsymbol{X}_{i},Y_{i},\bm \beta \right) |,$ and by {\bf C4}, we conclude

$$ \max_{1\leq i\leq n} | \boldsymbol{t}^T \left( \bm \beta \right) \boldsymbol{g}\left( \boldsymbol{X}_{i},Y_{i},\bm \beta \right) |= O_p (n^{-1/2}) o(n^{1/2}) = o_p(1).$$ 

Then, denoting $\gamma_i=\boldsymbol{t}^{T}\left(
\bm \beta_0 \right)  \boldsymbol{g}\left(  \boldsymbol{X}_{i},Y_{i},\bm \beta_0 \right) ,$ we have

\begin{eqnarray*}
0 & = & \frac{1}{n}\sum\limits_{i=1}^{n}\frac{\boldsymbol{g}\left(  \boldsymbol{X}_{i},Y_{i},\bm \beta_0 \right)  }{1+\boldsymbol{t}^{T}\left(
\bm \beta_0 \right)  \boldsymbol{g}\left(  \boldsymbol{X}_{i},Y_{i},\bm \beta_0 \right) } \\
 & = & \frac{1}{n}\sum\limits_{i=1}^{n}\boldsymbol{g}\left(  \boldsymbol{X}_{i},Y_{i},\bm \beta_0 \right)  (1- \gamma_i  + {\gamma_i^2 \over (1+\gamma_i)}) \\
 & = & \frac{1}{n}\sum\limits_{i=1}^{n}\boldsymbol{g}\left(  \boldsymbol{X}_{i},Y_{i},\bm \beta \right) - \boldsymbol{S}\boldsymbol{t}\left(
\bm \beta_0 \right) + {1\over n} \sum\limits_{i=1}^{n}\boldsymbol{g}\left(  \boldsymbol{X}_{i},Y_{i},\bm \beta \right) {\gamma_i^2 \over 1 +\gamma_i}. 
\end{eqnarray*}

Now, it can be seen in \cite{owe90} that given $Y_i\geq 0$ i.i.d. and $E[Y_i^2]<\infty ,$ then ${\displaystyle {1\over n} \sum_{i=1}^n Y_i^3 = o(n^{1/2}).}$ Defining $Y_i:= |\boldsymbol{g}\left( \boldsymbol{X}_{i},Y_{i},\bm \beta \right) |,$ and by {\bf C4}, we conclude that the final term is bounded by

$$  {1\over n} \sum\limits_{i=1}^{n} || \boldsymbol{g}\left(  \boldsymbol{X}_{i},Y_{i},\bm \beta \right) ||^3 ||\boldsymbol{t}\left(
\bm \beta_0 \right) ||^2 |1 + \gamma_i|^{-1} = o(n^{1/2}) O_p(n^{-1}) O_p(1) = o_p(n^{-1/2}).$$ 
Hence, we conclude that

\begin{eqnarray*}
\boldsymbol{t} (\bm \beta_{0}) & = & {1\over n}\left( {1\over n} \sum_{i=1}^n \boldsymbol{g}\left(  \boldsymbol{X}_{i},Y_{i},\bm \beta_{0} \right) \boldsymbol{g}^T\left(  \boldsymbol{X}_{i},Y_{i},\bm \beta_{0} \right) \right)^{-1} \sum_{i=1}^n \boldsymbol{g}\left(  \boldsymbol{X}_{i},Y_{i},\bm \beta_{0} \right) + o_p(n^{-1/2})\\
& = & \left( \sum_{i=1}^n \boldsymbol{g}\left(  \boldsymbol{X}_{i},Y_{i},\bm \beta_{0} \right) \boldsymbol{g}^T\left(  \boldsymbol{X}_{i},Y_{i},\bm \beta_{0} \right) \right)^{-1} \sum_{i=1}^n \boldsymbol{g}\left(  \boldsymbol{X}_{i},Y_{i},\bm \beta_{0} \right) + o_p(n^{-1/2}) \\
& = & (n\boldsymbol{S})^{-1} \sum_{i=1}^n \boldsymbol{g}\left(  \boldsymbol{X}_{i},Y_{i},\bm \beta_{0} \right) + o_p(n^{-1/2}).
\end{eqnarray*}

Therefore, the first term in \eqref{A} can be written as

$$ \sum_{i=1}^n \left( \left( n\boldsymbol{S}\right)^{-1} \sum_{i=1}^n \boldsymbol{g}^T \left(  \boldsymbol{X}_{i},Y_{i},\bm \beta_{0} \right) + o_p(n^{-1/2}) \right)^T (n\boldsymbol{S}) \left( \left( n\boldsymbol{S} \right)^{-1} \sum_{i=1}^n \boldsymbol{g}\left(  \boldsymbol{X}_{i},Y_{i},\bm \beta_{0} \right) + o_p(n^{-1/2}) \right) $$

$$ = \sum_{i=1}^n \boldsymbol{g}^T \left(  \boldsymbol{X}_{i},Y_{i},\bm \beta_{0} \right) \left( n\boldsymbol{S}\right)^{-1}  \sum_{i=1}^n \boldsymbol{g}\left(  \boldsymbol{X}_{i},Y_{i},\bm \beta_{0} \right) + o_p(n^{-1/2}) $$ 

$$ = \sqrt{n} \left( {1\over n} \sum_{i=1}^n \boldsymbol{g}^T \left(  \boldsymbol{X}_{i},Y_{i},\bm \beta_{0} \right) \right) \boldsymbol{S}^{-1}  \sqrt{n} \left( {1\over n}\sum_{i=1}^n \boldsymbol{g}\left(  \boldsymbol{X}_{i},Y_{i},\bm \beta_{0} \right) \right) + o_p(n^{-1/2}) .$$ 

Now, $\boldsymbol{S}$ defined in \eqref{defS} satisfies

$$ \boldsymbol{S} \underset{n\rightarrow\infty}{\overset{\mathcal{P}}{\longrightarrow }} E\left[ \boldsymbol{g}\left(  \boldsymbol{X},Y,\bm \beta_{0} \right) \boldsymbol{g}^T\left(  \boldsymbol{X},Y,\bm \beta_{0} \right) \right] ,$$ 
and applying the Central Limit Theorem, we get under the null hypothesis

$$ \sqrt{n} {1\over n} \sum_{i=1}^n \boldsymbol{g}\left(  \boldsymbol{X}_{i},Y_{i},\bm \beta_{0} \right) \underset{n\rightarrow\infty}{\overset{\mathcal{L}}{\longrightarrow }} {\cal N}\left( \boldsymbol{0}, E\left[ \boldsymbol{g}\left(  \boldsymbol{X},Y,\bm \beta_{0} \right) \boldsymbol{g}^T\left(  \boldsymbol{X},Y,\bm \beta_{0} \right) \right] \right) .$$ Therefore,

$$ o\left( \sum_{i=1}^n \boldsymbol{t}^T (\bm \beta_{0}) \boldsymbol{g}\left(  \boldsymbol{X}_{i},Y_{i},\bm \beta_{0} \right) \boldsymbol{g}^T\left(  \boldsymbol{X}_{i},Y_{i},\bm \beta_{0} \right) \boldsymbol{t} (\bm \beta_{0})\right) = o(O_p(1))= o_p(1),$$ 
and
$$ {2n d_{\phi }(\boldsymbol{u},\boldsymbol{p}(\bm \beta_{0}))\over \phi^{''}(1)} \underset{n\rightarrow\infty}{\overset{\mathcal{L}}{\longrightarrow }} \chi^2_{q}.$$
\end{proof}

Based on the previous theorem, we reject the null hypothesis given in (\ref{0.6}) if

\begin{equation}
T_{n}^{\phi}(\boldsymbol{\beta}) > \chi_{q,1-\alpha}^{2},
\end{equation}
and thus we can define and approximated confidence region for
$\boldsymbol{\beta}$ as follows:

\[ C_{1-\alpha}^{\ast}=\{\boldsymbol{\beta}\, \, \boldsymbol{/}\,\, T_{n}^{\phi}(\boldsymbol{\beta})< \chi_{q,1-\alpha}^{2}\} .\]


There are some classical measures of divergence which cannot be expressed as a $\phi $-divergence measure, such as the divergence measures of Battacharya \cite{bha43}, R\'enyi \cite{ren61}, and Sharma and Mittal \cite{shmi77}. However, such measures are particular cases of the $(h,\phi )${\it -divergence measures} defined by
\[
d_{\phi }^{h}\left( \boldsymbol{u},\boldsymbol{p}\left( \bm \beta_0\right) \right)
:=h\left( d_{\phi }\left( \boldsymbol{u},\boldsymbol{p}\left( \bm \beta_0\right) \right) \right) ,
\]
where $h$ is a differentiable increasing function mapping from $\left[
0,\infty \right) $ onto $\left[ 0,\infty \right) $, with $h(0)=0$ and $
h^{\prime }(0)>0$. In Table \ref{t1}, these particular divergence
measures are presented, together with the corresponding expressions of $h$ and $\phi $.

{\small
\begin{table}[htbp] \tabcolsep0.8pt  \centering%
$
\begin{tabular}{|c|ccc|c|}
\hline
Divergence & \hspace*{0.5cm} & $h\left( x\right) $ & \hspace*{0.5cm} & $\phi
\left( x\right) $ \\ \hline
R\'{e}nyi &  & $\frac{1}{a\left(
a-1\right) }\log \left( a\left( a-1\right) x+1\right) ,\quad a\neq 0,1$ &
& $\frac{x^{a}-a\left( x-1\right) -1}{a\left( a-1\right) }
,\quad a\neq 0,1$ \\
Sharma-Mittal &  & $\frac{1}{b-1}
\left\{ [1+a\left( a-1\right) x]^{\frac{b-1}{a-1}}-1\right\} ,\quad b,a\neq
1 $ &  & $\frac{x^{a}-a\left( x-1\right) -1}{a\left(
a-1\right) },\quad a\neq 0,1$ \\
Battacharya &  & $-\log \left(
-x+1\right) $ &  & $-x^{1/2}+\frac{1}{2}\left(
x+1\right) $ \\ \hline
\end{tabular}
\ \ \ \ \ \ \ \ \ \ \ \ \ \ \ $
\caption{Some specific $(h,\phi)$-divergence
measures.\label{t1}}
\end{table}
}

The $(h,\phi )$-divergence measures were introduced in \cite{memopasa95} and some associated asymptotic results for them were established in \cite{mepapa97}.

%
%
%

Now, the family of empirical phi-divergence test statistics defined previously can be extended to define the empirical $(h, \phi )$-divergence test statistics based on $(h,\phi )$-divergence measures.

\begin{definition}
If {\bf C2} and {\bf C3} hold, the family of empirical $(h, \phi )$-divergence test statistics for testing (\ref{0.6}) is defined by

$$ T_{n}^{\phi, h}\left( \bm \beta_0 \right):= {2n\over \phi^{''}(1)h'(0)} h(d_{\phi }\left( \boldsymbol{u},\boldsymbol{p}\left( \bm \beta_0\right) \right) $$
with $\phi \in \Phi^{\ast}$ and $h: \left[
0,\infty \right) \rightarrow \left[ 0,\infty \right) $ a differentiable function satisfying $h(0)=0$ and $
h^{\prime }(0)>0.$
\end{definition}

The asymptotic distribution of the empirical $(h, \phi )$-divergence test statistics $T_{n}^{\phi, h}\left( \bm \beta_0 \right) $ is given in next theorem.

\begin{theorem}
Under the conditions of Theorem \ref{theo1}, the asymptotic distribution of the family of empirical $(h,\phi )$-divergence test statistics $ T_{n}^{\phi, h}\left( \bm \beta_0 \right) $ is a $\chi^2_{q}.$
\end{theorem}

\begin{proof}
Note that we can apply a first order Taylor expansion to $h$ obtaining

$$ h(x)=h(0) + h'(0)x + o_p(x).$$
Therefore, as $h(0)=0$,

$$ h\left( d_{\phi }\left( \boldsymbol{u},\boldsymbol{p}\left( \bm \beta_0\right) \right) \right) = h'(0)d_{\phi }\left( \boldsymbol{u},\boldsymbol{p}\left( \bm \beta_0\right) \right)  + o_p\left( d_{\phi }\left( \boldsymbol{u},\boldsymbol{p}\left( \bm \beta_0\right) \right) \right) .$$

Thus, $ 2n {h\left( d_{\phi }\left( \boldsymbol{u},\boldsymbol{p}\left( \bm \beta_0\right) \right) \right)\over h'(0) \phi^{''}(1)}= T_{n}^{\phi}\left( \bm \beta_0 \right)$ and ${2nd_{\phi }\left( \boldsymbol{u},\boldsymbol{p}\left( \bm \beta_0\right) \right) \over \phi^{''}(1)}= T_{n}^{\phi, h}\left( \bm \beta_0 \right)$ have the same asymptotic distribution, and hence the result holds.
\end{proof}

\subsection{Power function}

In this section we present an approximation to the power function of the empirical phi-divergence test statistics introduced in (\ref{0.9}) for the null hypothesis considered in (\ref{0.6}). Thus, we look for the probability
$$ \Pi_{\phi }^n(\bm \beta^*) =P_{\bm \beta^*} (T_{n}^{\phi}(\bm \beta_{0}) > \chi^2_{q; \alpha }), \quad \bm \beta^* \ne \bm \beta_0.$$

For this, we first introduce a theorem that will be important for obtaining the asymptotic approximation of the power function.

\begin{theorem}
Assume conditions {\bf C1}-{\bf C9} hold. Let us assume that the logistic regression model holds and denote by $\bm \beta^*$ the true parameter of $\bm \beta $; suppose $\bm \beta^* \ne \bm \beta_0.$ Let us also suppose

\begin{itemize}
\item $ E_{\bm \beta^*} \left[ {\boldsymbol{g}(\boldsymbol{X}, Y, \bm \beta_0) \over 1 + \bm \tau^T \boldsymbol{g}(\boldsymbol{X}, Y, \bm \beta_0)}\right] $ exists and there exists $\bm \tau $ being the only solution of
$$ E_{\bm \beta^*} \left[ {\boldsymbol{g}(\boldsymbol{X}, Y, \bm \beta_0) \over 1 + \bm \tau^T \boldsymbol{g}(\boldsymbol{X}, Y, \bm \beta_0)}\right] =\boldsymbol{0}_q,$$
\item $ \boldsymbol{n}(\bm \beta_0, \bm \beta^*) := E_{\bm \beta^*} \left[ {-\boldsymbol{g}(\boldsymbol{X}, Y, \bm \beta_0)\boldsymbol{g}^T(\boldsymbol{X}, Y, \bm \beta_0) \over (1 + \bm \tau^T \boldsymbol{g}(\boldsymbol{X}, Y, \bm \beta_0))^2} \right] $ exists and has maximal rank.
\item $ \boldsymbol{m}(\bm \beta_0, \bm \beta^*):= E_{\bm \beta^*} \left[ {\boldsymbol{g}(\boldsymbol{X}, Y, \bm \beta_0 ) \over 1 + \bm \tau^T \boldsymbol{g}(\boldsymbol{X}, Y, \beta_0 )} \psi \left( 1 + \bm \tau^T  \boldsymbol{g}(\boldsymbol{X}, Y, \beta_0 )\right) \right] $
exists, where $\psi $ is given by
\begin{equation}\label{psi}
\psi (x):= {-1\over x} \phi (x) + \phi'(x).
\end{equation}
\end{itemize}

Then,

$$ \sqrt{n} \left( d_{\phi }\left( \boldsymbol{u},\boldsymbol{p}\left( \bm \beta_0 \right) \right) - d_{\phi }\left( \boldsymbol{u},\boldsymbol{p}\left( \bm \beta^* \right) \right) \right)  \overset{\mathcal{L}}{\underset{n\rightarrow \infty }\longrightarrow} N\left( \boldsymbol{0}, \sigma^2(\bm \beta_0, \bm \beta^* )\right) ,$$
with
$$ \sigma^2(\bm \beta_0, \bm \beta^* ) := \boldsymbol{m}^T(\bm \beta_0, \bm \beta^*) \boldsymbol{n}(\bm \beta_0, \bm \beta^*)^{-1}E_{\beta^*}\left[ {\boldsymbol{g}(\boldsymbol{X}, Y, \bm \beta_0) \over 1+ \bm \tau^t \boldsymbol{g}(\boldsymbol{X}, Y, \bm \beta_0)}\right] E_{\beta^*}\left[ {\boldsymbol{g}(\boldsymbol{X}, Y, \bm \beta_0) \over 1+ \bm \tau^t \boldsymbol{g}(\boldsymbol{X}, Y, \bm \beta_0)}\right]^T  \boldsymbol{n}(\bm \beta_0, \bm \beta^*)^{-1} \boldsymbol{m}(\bm \beta_0, \bm \beta^*) .$$
\end{theorem}

\begin{proof}
Let us denote by $\bm \tau $ the (only) solution of

$$ E_{\bm \beta^*} \left[ {\boldsymbol{g}(\boldsymbol{X}, Y, \bm \beta_0) \over 1 + \bm \tau^T \boldsymbol{g}(\boldsymbol{X}, Y, \bm \beta_0)}\right] =\boldsymbol{0}_q.$$

Now, let us denote by $\boldsymbol{t}(\bm \beta_0)$ the unique (by {\bf C2} and {\bf C3}) vector satisfying

\begin{equation}\label{eqaux1}
{1\over n} \sum_{i=1}^n {\boldsymbol{g}(\boldsymbol{X}_i, Y_i, \bm \beta_0) \over 1 + \boldsymbol{t}^T(\bm \beta_0) \boldsymbol{g}(\boldsymbol{X}_i, Y_i, \bm \beta_0)} = \boldsymbol{0}_q.
\end{equation}

Then, by the Weak Law of Large Numbers, it is clear that

$$ \boldsymbol{t}(\bm \beta_0)  \overset{\mathcal{P}}{\underset{n\rightarrow \infty }\longrightarrow} \bm \tau .$$

We rename $d_{\phi }\left( \boldsymbol{u},\boldsymbol{p}\left( \bm \beta_0 \right) \right) $ as a function depending on $\boldsymbol{t}(\bm \beta_0 )$ denoted $D_{\phi }(\boldsymbol{t}(\bm \beta_0 )).$ Then,

$$ D_{\phi}(\boldsymbol{t}(\bm \beta_{0}))= {1\over n} \sum_{i=1}^n {1\over 1 + \boldsymbol{t}^T(\bm \beta_0) \boldsymbol{g}(\boldsymbol{X}_i, Y_i, \bm \beta_0)} \phi \left( 1 + \boldsymbol{t}^T(\bm \beta_0) \boldsymbol{g}(\boldsymbol{X}_i, Y_i, \bm \beta_0)\right) .$$

A first order Taylor expansion of $D_{\phi}(\boldsymbol{t}(\bm \beta_{0}))$ around $\boldsymbol{t}(\bm \beta_0 )=\bm \tau $ leads to

$$ D_{\phi}(\boldsymbol{t}(\bm \beta_{0})) = D_{\phi}(\bm \tau ) + \left( {\partial D_{\phi }(\boldsymbol{t}(\bm \beta_0 ))\over \partial \boldsymbol{t}(\bm \beta_0 )} \right)_{\boldsymbol{t}(\bm \beta_0 )=\bm \tau}^T (\boldsymbol{t}(\bm \beta_{0}) - \bm \tau ) + o(\| \boldsymbol{t}(\bm \beta_{0}) - \bm \tau \|).$$

On the other hand, $\left( {\partial D_{\phi }(\boldsymbol{t}(\bm \beta_0 ))\over \partial \boldsymbol{t}(\bm \beta_0 )} \right) $ is given by

{\small
$$ {1\over n} \sum_{i=1}^n \left\{  {-\boldsymbol{g}(\boldsymbol{X}_i, Y_i, \bm \beta_0 )\over (1 + \boldsymbol{t}^T(\bm \beta_0 ) \boldsymbol{g}(\boldsymbol{X}_i, Y_i, \bm \beta_0 ))^2} \phi \left( 1 + \boldsymbol{t}^T(\bm \beta_0 ) \boldsymbol{g}(\boldsymbol{X}_i, Y_i, \bm \beta_0 )\right) + {\boldsymbol{g}(\boldsymbol{X}_i, Y_i, \bm \beta_0 )\over 1 + \boldsymbol{t}^T(\bm \beta_0 ) \boldsymbol{g}(\boldsymbol{X}_i, Y_i, \bm \beta_0 )} \phi ' \left( 1 + \boldsymbol{t}^T(\bm \beta_0) \boldsymbol{g}(\boldsymbol{X}_i, Y_i, \bm \beta_0)\right) \right\} $$
$$ = {1\over n}\sum_{i=1}^n {\boldsymbol{g}(\boldsymbol{X}_i, Y_i, \bm \beta_0 )\over 1 + \boldsymbol{t}^T(\bm \beta_0 ) \boldsymbol{g}(\boldsymbol{X}_i, Y_i, \bm \beta_0 )}\left\{  {-1\over 1 + \boldsymbol{t}^T(\bm \beta_0 ) \boldsymbol{g}(\boldsymbol{X}_i, Y_i, \bm \beta_0 )} \phi \left( 1 + \boldsymbol{t}^T(\bm \beta_0 ) \boldsymbol{g}(\boldsymbol{X}_i, Y_i, \bm \beta_0) \right)  + \phi ' \left( 1 + \boldsymbol{t}^T(\bm \beta_0 ) \boldsymbol{g}(\boldsymbol{X}_i, Y_i, \bm \beta_0) \right) \right\} $$}
$$ = {1\over n}\sum_{i=1}^n {\boldsymbol{g}(\boldsymbol{X}_i, Y_i, \bm \beta_0 )\over 1 + \boldsymbol{t}^T(\bm \beta_0 ) \boldsymbol{g}(\boldsymbol{X}_i, Y_i, \bm \beta_0 )} \psi \left( 1 + \boldsymbol{t}^T(\bm \beta_0) \boldsymbol{g}(\boldsymbol{X}_i, Y_i, \bm \beta_0)\right) ,$$
where $\psi (x)$ is defined in \eqref{psi}. Consequently, by the Weak Law of Large Numbers,

$$ \left( {\partial D_{\phi }(\boldsymbol{t}(\bm \beta_0 ))\over \partial \boldsymbol{t}(\bm \beta_0 )} \right)_{\boldsymbol{t}(\bm \beta_0 )=\bm \tau}  \overset{\mathcal{P}}{\underset{n\rightarrow \infty }\longrightarrow} E_{\bm \beta^*} \left[ {\boldsymbol{g}(\boldsymbol{X}, Y, \bm \beta_0 ) \over 1 + \bm \tau^T \boldsymbol{g}(\boldsymbol{X}, Y, \bm \beta_0 )} \psi \left( 1 + \bm \tau^T \boldsymbol{g}(\boldsymbol{X}, Y, \bm \beta_0 )\right) \right]  =: \boldsymbol{m}(\bm \beta_0 , \bm \beta^*).$$

Let us now consider the function

$$ \boldsymbol{h}(\boldsymbol{t}(\bm \beta )) := {1\over n} \sum_{i=1}^n {\boldsymbol{g}(\boldsymbol{X}_i, Y_i, \bm \beta )\over 1 + \boldsymbol{t}^T(\bm \beta ) \boldsymbol{g}(\boldsymbol{X}_i, Y_i, \bm \beta )}.$$

Remark that $h(\boldsymbol{t}(\bm \beta_0)) =\boldsymbol{0}_q$ by \eqref{eqaux1}. A first order Taylor expansion of $h(\boldsymbol{t}(\bm \beta ))$ around $\bm \tau $ leads to

$$ \boldsymbol{h}(\boldsymbol{t}(\bm \beta_{0})) = h(\bm \tau ) + \left( {\partial h(\boldsymbol{t}(\bm \beta ))\over \partial \boldsymbol{t}(\bm \beta )} \right)^T_{\boldsymbol{t}(\bm \beta )=\bm \tau} (\boldsymbol{t}(\bm \beta_{0}) - \bm \tau ) + o(\| \boldsymbol{t}(\bm \beta_{0}) - \bm \tau \| ),$$ so that

$$  (\boldsymbol{t}(\bm \beta_{0}) - \bm \tau )= -\left( \left( {\partial h(\boldsymbol{t}(\bm \beta ))\over \partial \boldsymbol{t}(\bm \beta )} \right)_{\boldsymbol{t}(\bm \beta )=\bm \tau}^{-1}\right)^{-1} h(\bm \tau ) + o(\| \boldsymbol{t}(\bm \beta_{0}) - \bm \tau \| ).$$

On the other hand,

$$\left( {\partial h(\boldsymbol{t}(\bm \beta_{0}))\over \partial \boldsymbol{t}(\bm \beta_0)} \right) = {1\over n} \sum_{i=1}^n {-\boldsymbol{g}(\boldsymbol{X}_i, Y_i, \bm \beta_0) \boldsymbol{g}^T(\boldsymbol{X}_i, Y_i, \bm \beta_0) \over (1 + \boldsymbol{t}^T(\bm \beta_0) \boldsymbol{g}(\boldsymbol{X}_i, Y_i, \bm \beta_0))^2}, $$
and hence, by the Weak Law of Large Numbers

$$ \left( {\partial h(\boldsymbol{t}(\bm \beta_{0}))\over \partial \boldsymbol{t}(\bm \beta_{0})} \right)_{\boldsymbol{t}(\bm \beta_{0})=\bm \tau}  \overset{\mathcal{P}}{\underset{n\rightarrow \infty }\longrightarrow} E_{\bm \beta^*} \left[ {-\boldsymbol{g}(\boldsymbol{X}, Y, \bm \beta_0)\boldsymbol{g}^T(\boldsymbol{X}, Y, \bm \beta_0) \over (1 + \bm \tau^T \boldsymbol{g}(\boldsymbol{X}, Y, \bm \beta_0))^2} \right]  =: \boldsymbol{n}(\bm \beta_0, \bm \beta^*).$$

As $\boldsymbol{n}^T(\bm \beta_0, \bm \beta^*)$ is a squared matrix of maximal rank, it can be inverted and consequently,

$$  (\boldsymbol{t}(\bm \beta_{0}) - \bm \tau )= \left( \boldsymbol{n}^T(\bm \beta_0, \bm \beta^*)\right)^{-1} {1\over n} \sum_{i=1}^n {\boldsymbol{g}(\boldsymbol{X}_i, Y_i, \bm \beta_0 )\over 1 + \bm \tau^T \boldsymbol{g}(\boldsymbol{X}_i, Y_i, \bm \beta_0 )} + o_p(1).$$

Therefore, applying the Central Limit Theorem,

$$ \sqrt{n} (\boldsymbol{t}(\bm \beta_{0}) - \bm \tau ) \overset{\mathcal{L}}{\underset{n\rightarrow \infty }\longrightarrow} N\left( \boldsymbol{0}, \left(\boldsymbol{n}^T(\bm \beta_0, \bm \beta^*)\right)^{-1}E_{\beta^*}\left[ {\boldsymbol{g}(\boldsymbol{X}, Y, \bm \beta_0) \over 1+ \bm \tau^T \boldsymbol{g}(\boldsymbol{X}, Y, \bm \beta_0)}\right] E_{\beta^*}\left[ {\boldsymbol{g}(\boldsymbol{X}, Y, \bm \beta_0) \over 1+ \bm \tau^T \boldsymbol{g}(\boldsymbol{X}, Y, \bm \beta_0)}\right]^T  \boldsymbol{n}(\bm \beta_0, \bm \beta^*)^{-1} \right) .$$

Finally,

$$ \sqrt{n}(D_{\phi}(\boldsymbol{t}(\bm \beta_{0})) - D_{\phi}(\bm \tau )) = \boldsymbol{m}^T(\bm \beta_0, \bm \beta^*) \sqrt{n}(\boldsymbol{t}(\bm \beta_{0}) - \bm \tau ) + o(\sqrt{n} \| \boldsymbol{t}(\bm \beta_0) - \bm \tau \| ).$$
Note that $ o(\sqrt{n} \| \boldsymbol{t}(\bm \beta_0) - \bm \tau \| )=o_P(1)$ and thus,

$$ \sqrt{n}(D_{\phi}(\boldsymbol{t}(\bm \beta_{0})) - D_{\phi}(\bm \tau )) \overset{\mathcal{L}}{\underset{n\rightarrow \infty }\longrightarrow} N\left( \boldsymbol{0}, \sigma (\bm \beta^*, \bm \beta_0 )\right) ,$$
where

$$ \sigma (\bm \beta^*, \bm \beta_0 ) := \boldsymbol{m}^T(\bm \beta_0, \bm \beta^*) \boldsymbol{n}(\bm \beta_0, \bm \beta^*)^{-1}E_{\beta^*}\left[ {\boldsymbol{g}(\boldsymbol{X}, Y, \bm \beta) \over 1+ \bm \tau^t \boldsymbol{g}(\boldsymbol{X}, Y, \bm \beta_0)}\right] E_{\beta^*}\left[ {\boldsymbol{g}(\boldsymbol{X}, Y, \bm \beta) \over 1+ \bm \tau^t \boldsymbol{g}(\boldsymbol{X}, Y, \bm \beta_0)}\right]^T  \boldsymbol{n}(\bm \beta_0, \bm \beta^*)^{-1} \boldsymbol{m}(\bm \beta_0, \bm \beta^*) ,$$
or equivalently,

$$ \sqrt{n} (d_{\phi }\left( \boldsymbol{u},\boldsymbol{p}\left( \bm \beta \right) \right) - d_{\phi }\left( \boldsymbol{u},\boldsymbol{p}\left( \bm \tau \right) \right) )  \overset{\mathcal{L}}{\underset{n\rightarrow \infty }\longrightarrow} N\left( \boldsymbol{0}, \sigma (\bm \beta^*, \bm \beta_0 ) \right) .$$
\end{proof}

Based on the previous result, let us obtain an approximation of the power function in $\bm \beta^* $. We have

\begin{eqnarray*}
\Pi_{\phi}^n(\bm \beta^*) & = & P_{\bm \beta = \bm \beta^* }\left( T_n^{\phi }(\bm \beta_0) > \chi^2_{q; \alpha }\right) \\
& = & P_{\bm \beta = \bm \beta^* }\left( {2n\over \phi ''(1)} d_{\phi }\left( \boldsymbol{u}, \boldsymbol{p}(\bm \beta_0 ) \right) > \chi^2_{q; \alpha }\right) \\
& = & P_{\bm \beta = \bm \beta^* }\left( \sqrt{n} \left( d_{\phi }\left( \boldsymbol{u}, \boldsymbol{p}(\bm \beta_0 ) \right) - d_{\phi }\left( \boldsymbol{u}, \boldsymbol{p}(\bm \tau ) \right) \right) > \sqrt{n}\left( {\chi^2_{q; \alpha }\phi''(1)\over 2n} -  d_{\phi }\left( \boldsymbol{u}, \boldsymbol{p}(\bm \tau ) \right) \right) \right) \\
& {\underset{n\rightarrow \infty }\longrightarrow} & 1- \Phi \left( \sqrt{n\over \sigma (\bm \beta^*, \bm \beta_0)}\left( {\chi^2_{q; \alpha }\phi''(1)\over 2n} -  d_{\phi }\left( \boldsymbol{u}, \boldsymbol{p}(\bm \tau ) \right) \right) \right) ,
\end{eqnarray*}
where $\Phi $ denotes the standard normal distribution function. As a consequence, if some alternative $\bm \beta^* \ne \bm \beta_0$ is the true parameter value, the probability of rejecting $\bm \beta_0$ with the rejection rule $T_n^{\phi } (\bm \beta_0)> \chi^2_{q;\alpha }$ for a fixed significance level $\alpha $ tends to 1 as $n\rightarrow \infty .$ Thus, the test is consistent in the sense of \cite{fra57}. Note that $\Pi_{\phi}^n(\bm \beta^*) $ depends on
$ d_{\phi }\left( \boldsymbol{u}, \boldsymbol{p}(\bm \tau ) \right) ,$ so in this case (tiny) differences might appear for different $\phi $.

Moreover, the previous result can be used to sort out an approximation of the minimum sample size needed for the empirical phi-divergence test statistic to have a given power $\Pi_{\phi }^n (\bm \beta^*)=: \pi^*$ and size $\alpha $. For this, notice that

$$ \pi^* \approx 1- \Phi \left( {\sqrt{n\over \sigma (\bm \beta^*, \bm \beta_0)}} \left( {\chi^2_{q; \alpha } \phi''(1)\over 2n} - d_{\phi }\left( \boldsymbol{u}, \boldsymbol{p}(\bm \tau )\right) \right) \right) \Leftrightarrow $$

\begin{eqnarray*}
\Phi^{-1}(1-\pi^*) & \approx & {\sqrt{n\over \sigma (\bm \beta^*, \bm \beta_0)}} \left( {\chi^2_{q; \alpha } \phi''(1)\over 2n} - d_{\phi }\left( \boldsymbol{u}, \boldsymbol{p}(\bm \tau )\right) \right) \\
& = & {\chi^2_{q; \alpha } \phi''(1)\over 2\sqrt{n} \sqrt{\sigma (\bm \beta^*, \bm \beta_0)}} - {\sqrt{n}\over \sqrt{\sigma (\bm \beta^*, \bm \beta_0)}} d_{\phi }\left( \boldsymbol{u}, \boldsymbol{p}(\bm \tau )\right) .
\end{eqnarray*}

Therefore,

$$ \Phi^{-2}(1-\pi^*) \approx  {\left( \chi^2_{q; \alpha } \phi''(1)\right)^2 \over 4n\sigma (\bm \beta^*, \bm \beta_0)} + {n\over \sigma (\bm \beta^*, \bm \beta_0)} d_{\phi }^2\left( \boldsymbol{u}, \boldsymbol{p}(\bm \tau )\right) - {\chi^2_{q; \alpha } \phi''(1) d_{\phi }\left( \boldsymbol{u}, \boldsymbol{p}(\bm \tau )\right) \over \sigma (\bm \beta^*, \bm \beta_0)}\Leftrightarrow   $$

$$ 4n\sigma (\bm \beta^*, \bm \beta_0) \Phi^{-2}(1-\pi^*) = \left( \chi^2_{q; \alpha } \phi''(1)\right)^2 + 4n^2 d_{\phi }^2\left( \boldsymbol{u}, \boldsymbol{p}(\bm \tau )\right) - 4n \chi^2_{q; \alpha } \phi''(1) d_{\phi }\left( \boldsymbol{u}, \boldsymbol{p}(\bm \tau )\right) \Leftrightarrow $$
$$ 4n^2 d_{\phi }^2\left( \boldsymbol{u}, \boldsymbol{p}(\bm \tau )\right) -4n \left[ \chi^2_{q; \alpha } \phi''(1) d_{\phi }\left( \boldsymbol{u}, \boldsymbol{p}(\bm \tau )\right) +  \sigma (\bm \beta^*, \bm \beta_0) \Phi^{-2}(1-\pi^*)  \right] + \left( \chi^2_{q; \alpha } \phi''(1)\right)^2 =0.$$
Thus,

\begin{equation}\label{H}
n= { l + \sqrt{l^2- 16 \left( \chi^2_{q; \alpha } \phi''(1)\right)^2 d_{\phi }^2\left( \boldsymbol{u}, \boldsymbol{p}(\bm \tau )\right) } \over 8d_{\phi }^2\left( \boldsymbol{u}, \boldsymbol{p}(\bm \tau )\right) },
\end{equation}
where $l= 4\left[ \chi^2_{q; \alpha } \phi''(1) d_{\phi }\left( \boldsymbol{u}, \boldsymbol{p}(\bm \tau )\right) +  \sigma (\bm \beta^*, \bm \beta_0) \Phi^{-2}(1-\pi^*) \right] $. Now,

{\small
\begin{eqnarray*}
l^2- 16 \left( \chi^2_{q; \alpha } \phi''(1)\right)^2 d_{\phi }^2\left( \boldsymbol{u}, \boldsymbol{p}(\bm \tau )\right) & = & 16 \left( \chi^2_{q; \alpha } \phi''(1)\right)^2 d_{\phi }^2\left( \boldsymbol{u}, \boldsymbol{p}(\bm \tau )\right) + 32 \left( \chi^2_{q; \alpha } \phi''(1)\right) d_{\phi }\left( \boldsymbol{u}, \boldsymbol{p}(\bm \tau )\right) \sigma (\bm \beta^*, \bm \beta_0) \Phi^{-2}(1-\pi^*) \\ & & + 16 \sigma^2 (\bm \beta^*, \bm \beta_0) \Phi^{-4}(1-\pi^*) -16 \left( \chi^2_{q; \alpha } \phi''(1)\right)^2 d_{\phi }^2\left( \boldsymbol{u}, \boldsymbol{p}(\bm \tau )\right) \\
& = & 16 \sigma^2 (\bm \beta^*, \bm \beta_0) \Phi^{-4}(1-\pi^*) + 32 \left( \chi^2_{q; \alpha } \phi''(1)\right) d_{\phi }\left( \boldsymbol{u}, \boldsymbol{p}(\bm \tau )\right) \sigma (\bm \beta^*, \bm \beta_0) \Phi^{-2}(1-\pi^*) \\
& = & 16 \sigma (\bm \beta^*, \bm \beta_0) \Phi^{-2}(1-\pi^*) \left[ \sigma (\bm \beta^*, \bm \beta_0) \Phi^{-2}(1-\pi^*) + 2 \chi^2_{q; \alpha } \phi''(1) d_{\phi }\left( \boldsymbol{u}, \boldsymbol{p}(\bm \tau )\right) \right] .
\end{eqnarray*}}

Hence, \eqref{H} is given by
\begin{eqnarray*}
 n & = & { 4\left[ \chi^2_{q; \alpha } \phi''(1) d_{\phi }\left( \boldsymbol{u}, \boldsymbol{p}(\bm \tau )\right) +  \sigma (\bm \beta^*, \bm \beta_0) \Phi^{-2}(1-\pi^*) \right] \over 8d_{\phi }^2\left( \boldsymbol{u}, \boldsymbol{p}(\bm \tau )\right) } \\
& + & {4 \sqrt{\sigma (\bm \beta^*, \bm \beta_0) \Phi^{-2}(1-\pi^*)} \sqrt{\sigma (\bm \beta^*, \bm \beta_0) \Phi^{-2}(1-\pi^*) +2 \chi^2_{q; \alpha } \phi''(1) d_{\phi }\left( \boldsymbol{u}, \boldsymbol{p}(\bm \tau )\right) }\over 8d_{\phi }^2\left( \boldsymbol{u}, \boldsymbol{p}(\bm \tau )\right) } .
\end{eqnarray*}

If we denote

$$ A:= 2\chi^2_{q; \alpha } \phi''(1) d_{\phi }\left( \boldsymbol{u}, \boldsymbol{p}(\bm \tau )\right) ,\, B= \sigma (\bm \beta^*, \bm \beta_0) \Phi^{-2}(1-\pi^*) ,$$ it follows

$$ n= {2A + 4B + 4 \sqrt{B(B+A)} \over  8d_{\phi }^2\left( \boldsymbol{u}, \boldsymbol{p}(\bm \tau )\right) } ={ A + 2B + 2 \sqrt{B(B+A)} \over  4d_{\phi }^2\left( \boldsymbol{u}, \boldsymbol{p}(\bm \tau )\right) } .$$

Finally, the desired sample size is

$$ n^* =[n] +1,$$ where $[x]$ denotes the integer part of $x$.

\section{Simulation study}

As explained before, empirical phi-divergence test statistics provide a family of empirical test statistics including the classical empirical likelihood ratio test. Moreover, all members of this family share the same asymptotic distribution. However, possible differences among them could arise for small and moderate sample sizes in terms of speed of convergence to the asymptotic distribution.
To compare the performance in these circumstances, we have carried out a simulation study; for this, we have considered the example proposed in \cite{roc03}. Let us explain this study: we have a logistic regression model with just one explanatory random variable $X$, that follows a standard Gaussian distribution. Then,

$$ logit\left( P(Y=1)\right) = \beta_0 + \beta_1 X. $$
Several values for $\beta_0, \beta_1$ are considered, thus obtaining four different models; this also allows to study the behavior when the overall probability of $P(Y=1)$ varies. Remark that this value is given by

$$ P(Y=1)= \int_{\mathbb{R}} {e^{\beta_0 + \beta_1 x}\over 1+e^{\beta_0 + \beta_1 x}} f_{{\cal N}(0,1)}(x) dx. $$

All these values appear in Table \ref{table1}.

\begin{table}[h]
\begin{center}
\begin{tabular}{|cccc|}
\hline Model & $P(Y=1)$ & $\beta_0$ & $\beta_1$ \\
\hline 1     & 0.5      & 0.00      & 4.36      \\
       2     & 0.4      & -1.16     & 4.20      \\
       3     & 0.3      & -2.16     & 3.71      \\
       4     & 0.2      & -2.80     & 2.82      \\
\hline
\end{tabular}
\end{center}
\caption{Model parameters.}
\label{table1}
\end{table}

We shall consider in our simulation study the family of $\phi -$divergences measures introduced in \cite{crre84}. This family of divergence measures, called the {\it power-divergence family}, is given by

\begin{equation}\label{eq9}
\phi (x)\equiv \phi_a(x)=\left\{ \begin{array}{lc} {1\over a(a+1)} (x^a- x- a(x-1)) & a\ne 0, a\ne -1 \\ x\log x -x+1 & a=0 \\ -\log x +x-1 & a=-1\end{array}\right. .
\end{equation}

Based on \eqref{0.9}, we get that the empirical power-divergence test statistics are given by

\begin{equation}\label{eq10}
T_{n}^{\phi_a}(\bm \beta_{0})={2n\over \phi^{\prime \prime }(1)} \sum_{i=1}^n \phi_a \left( 1+ \boldsymbol{t}^t (\bm \beta_{0}) \boldsymbol{g}\left(  \boldsymbol{X}_{i},Y_{i},\bm \beta_{0} \right) \right) {1\over 1+\boldsymbol{t}^t (\bm \beta_{0}) \boldsymbol{g}\left(  \boldsymbol{X}_{i},Y_{i},\bm \beta_{0} \right) }.
\end{equation}

For $a=0$, we recover the ELRT given by Eq. \eqref{0.7}. For a fixed model $M_i, i=1,2,3,4$, we proceed as follows: In order to compare the different empirical power-divergence test statistics, we have considered the following sample sizes $n=50, 100, 200, 400, 800$. Each sample consists in $n$ random vectors $(x_{ij}, Y_{ij}), i=1, 2, 3, 4, j=1, ..., n$ according to the parameters given in Table \ref{table1}. Then, we have to obtain $\bm t(\bm \beta_0).$ Note that in this case $\bm t(\bm \beta_0)$ has two components, that we will denote $t_0, t_1.$ To compute such values, we have to solve \eqref{5bis}. We have applied the Fourier-Motzkin elimination method in order to obtain bounds for $t_1$ so that $t_1 \in (t_1^-, t_1^+)$ and the same for $t_2$ (depending on $t_1$). To avoid numerical problems arising in the resolution of \eqref{5bis}, we have applied the hybrid method of Powell with analytic jacobian, starting with a fixed number $N_{pun}$ of values for $t_1$ in $(t_1^-, t_1^+)$ and then, for each of these values, a value for $t_2,$ so that a solution with a tolerance of $\epsilon $ is achieved. We have considered $N_{pun}=50$ and $\epsilon =10^{-10}.$

Next, we compute the value of the test statistic in Eq. \eqref{eq10} for the following values of $a$: -1, -0.5, -0.25, -0.125, 0, 0.5, 1, 0.67, 1, 1.5, 3. The values of the empirical power-divergence test statistics are compared to two nominal levels $1-\alpha =0.90, 0.95;$ according to the asymptotic distribution, we accept $H_0$ if the value of the empirical power-divergence test statistics do not exceed $\chi^2_{2; 1-\alpha }$. And we repeat this $N=1 000$ times for each combination $M_i$/$n$/$a$/$1-\alpha $.

We count the proportion of times for which $H_0$ is concluded in the 1000 experiences. According to the theory developed before, if $H_0$ holds, this proportion should be near the theoretical signification level $1-\alpha $ if the empirical power-divergence test statistics are near the asymptotic distribution $\chi^2_2$; this approach will be called $\chi^2$-approximation.

On the other hand, Owen \cite{owe90} has proposed to replace the $\chi^2_q$ asymptotic distribution with a $(n-1)q/(n-q)F_{q, n-q}$ distribution for small and moderate sample sizes; in the following, we will denote it as the $F$-approximation.


The results appear in Tables \ref{tab1} and \ref{tabla2}. In order to lighten the tables up, we do not write the sample sizes of 400 and 800 as the results are similar to 200. As it can be seen, the accuracy of the power $\phi $-divergence test statistics increases when the sample size grows, in accordance to the theoretical results; besides, it seems that the $F$ approximation has a higher coverage than the $\chi^2$ approximation; as both coverages are always under the significance level, it may be concluded that the $F$ approximation works better than the $\chi^2$ approach.

In the tables we have enlighten the values near the theoretical level. In this sense, a popular option for ``near" has been proposed by Dale \cite{dal86}, where the simulated signification $\hat{\alpha }$ is considered near if it satisfies $|\mathrm{logit}(1-\hat{\alpha})-\mathrm{logit}(1-\alpha)|\leq d$ and $d\in(0.35,0.70)$. We have
chosen $d=0.35$, so that we only take under consideration the power-divergence test statistics such that the corresponding exact size is
\begin{equation}
\hat{\alpha}_{N}^{a}\in(0.073,0.136)\label{eq24}%
\end{equation}
for the nominal size $\alpha=0.1$ and
\begin{equation}
\hat{\alpha}_{N}^{a}\in(0.036,0.070)\label{eq23}%
\end{equation}
for $\alpha=0.05$. Then, we only consider values in $(0.864, 0.927)$ for $1-\alpha =0.90$ and $(0.930, 0.964)$ for $1-\alpha =0.95.$ These values are written in bold in Tables \ref{tab1} and \ref{tabla2}.

Comparing the results for the different values of $a$, it can be seen that positive values of $a$ (0.5, 0.67, 1, 1.5) present better behavior than the ELRT ($a=0$); on the other hand, negative values of $a$ behave worse than the ELRT, although differences are tiny; this is in consonance with similar results arising in other different situations where power-divergence measures have been applied (see e.g. \cite{par06}). Indeed, if we see the tables as a function $f(a),$ it seems that $f$ attains a maximum in the interval $[0.67, 1.5].$ Thus, we may conclude that there are some values for $a$ different from that could compete with the ELRT.

\begin{table}[h]
\begin{center}
\begin{tabular}{|c|c|cccccccccc|}
\hline Model 1 & $a$      & -1    & -0.5  & -0.25 & -0.125 & 0     & 0.5   & 0.67  & 1     & 1.5   & 3 \\
\hline $n=50$  & $\chi^2$ & 0.768 &  0.779&  0.780&  0.780 &  0.781&  0.788&  0.793&  0.787&  0.787&  0.743\\
               & $F$      &  0.781&  0.788&  0.795&  0.795 &  0.798&  0.808&  0.807&  0.806&  0.796&  0.755\\
\hline $n=100$ & $\chi^2$ &  0.845& 0.852 & 0.852 & 0.857 &  0.862 &  0.863 & 0.863 &  {\bf 0.870}&  {\bf 0.866}&  0.847\\
               & $F$      &  0.854 & 0.862 &  {\bf 0.867}&  {\bf 0.870} &  {\bf 0.871}&  {\bf 0.874}&  {\bf 0.873}&  {\bf 0.874}&  {\bf 0.878}&  0.852 \\
\hline $n=200$ & $\chi^2$ &  0.860 &  {\bf 0.864}&  0.861 &  0.862 &  0.863 &  {\bf 0.869}&  {\bf 0.868}&  {\bf 0.871}&  {\bf 0.868}&  0.849 \\
               & $F$      &  0.861 &  {\bf 0.868}&  {\bf 0.864}&  {\bf 0.865} &  {\bf 0.866}&  {\bf 0.872}&  {\bf 0.875}&  {\bf 0.877}&  {\bf 0.873}& 0.853 \\

\hline Model 2 & $a$      & -1    & -0.5  & -0.25 & -0.125 &  0    &  0.5  &  0.67 &  1    &  1.5  &  3 \\
\hline $n=50$  & $\chi^2$ &  0.768&  0.786&  0.789&  0.790 &  0.789&  0.788&  0.786&  0.789&  0.781&  0.734 \\
               & $F$      &  0.781&  0.799&  0.805&  0.809 &  0.810&  0.808&  0.807&  0.805&  0.792&  0.756 \\
\hline $n=100$ & $\chi^2$ &  0.799&  0.818&  0.826&  0.827 &  0.834&  0.840&  0.837&  0.836&  0.834&  0.809 \\
               & $F$      &  0.810&  0.828&  0.836&  0.839 &  0.840&  0.848&  0.847&  0.846&  0.844&  0.813 \\
\hline $n=200$ & $\chi^2$ &  {\bf 0.868}&  {\bf 0.872}&  {\bf 0.875}&  {\bf 0.874} &  {\bf 0.874}&  {\bf 0.882}&  {\bf 0.884}&  {\bf 0.890}&  {\bf 0.893}&  {\bf 0.872} \\
               & $F$      &  {\bf 0.869}&  {\bf 0.879}&  {\bf 0.877}&  {\bf 0.878} &  {\bf 0.877}&  {\bf 0.887}&  {\bf 0.891}&  {\bf 0.892}&  {\bf 0.894}&  {\bf 0.875} \\

\hline Model 3 & $a$      &-1     & -0.5  & -0.25 & -0.125 &  0    &  0.5  &  0.67 &  1    &  1.5  &  3 \\
\hline $n=50$  & $\chi^2$ &  0.756&  0.760&  0.768&  0.770 &  0.770&  0.780&  0.781&  0.780&  0.770&  0.734 \\
               & $F$      &  0.768&  0.779&  0.780&  0.784 &  0.788&  0.798&  0.793&  0.797&  0.786&  0.747 \\
\hline $n=100$ & $\chi^2$ &  0.820&  0.835&  0.838&  0.843 &  0.843&  0.846&  0.844&  0.846&  0.845&  0.816\\
               & $F$      &  0.828&  0.841&  0.849&  0.854 &  0.859&  0.855&  0.854&  0.854&  0.855&  0.818 \\
\hline $n=200$ & $\chi^2$ &  {\bf 0.866}&  {\bf 0.874}&  {\bf 0.876}&  {\bf 0.877} &  {\bf 0.878}&  {\bf 0.885}&  {\bf 0.885}&  {\bf 0.888}&  {\bf 0.884}&  {\bf 0.874} \\
               & $F$      &  {\bf 0.869}&  {\bf 0.876}&  {\bf 0.879}&  {\bf 0.880} &  {\bf 0.883}&  {\bf 0.887}&  {\bf 0.888}&  {\bf 0.889}&  {\bf 0.887}&  {\bf 0.876} \\

\hline Model 4 & $a$      &-1     & -0.5  & -0.25 & -0.125 &  0    &  0.5  &  0.67 &  1    &  1.5  &  3 \\
\hline $n=50$  & $\chi^2$ & 0.746 &  0.764&  0.772&  0.775 &  0.775&  0.790&  0.788&  0.790&  0.789&  0.740 \\
               & $F$      &  0.765&  0.779&  0.785&  0.791 &  0.796&  0.806&  0.809&  0.809&  0.805&  0.763 \\
\hline $n=100$ & $\chi^2$ &  0.807&  0.825&  0.832&  0.838 &  0.840&  0.843&  0.847&  0.847&  0.848&  0.841 \\
               & $F$      &  0.816&  0.836&  0.845&  0.845 &  0.845&  0.852&  0.851&  0.858&  0.859&  0.848 \\
\hline $n=200$ & $\chi^2$ &  {\bf 0.864}&  {\bf 0.870}&  {\bf 0.871}&  {\bf 0.879} &  {\bf 0.880}&  {\bf 0.880}&  {\bf 0.881}&  {\bf 0.880}&  {\bf 0.885}&  {\bf 0.877} \\
               & $F$      &  {\bf 0.867}&  {\bf 0.874}&  {\bf 0.876}&  {\bf 0.882} &  {\bf 0.882}&  {\bf 0.887}&  {\bf 0.888}&  {\bf 0.886}&  {\bf 0.892}&  {\bf 0.880} \\
\hline
\end{tabular}
\end{center}

\caption{Values of the test statistics for different parameters $a$ for $1-\alpha =0.90$.}
\label{tab1}
\end{table}

\begin{table}[h]
\begin{center}
\begin{tabular}{|c|c|cccccccccc|}
\hline Model 1 & $a$      & -1    & -0.5  & -0.25 & -0.125 & 0     & 0.5   & 0.67  & 1     & 1.5   & 3 \\
\hline $n=50$  & $\chi^2$ &  0.813&  0.824&  0.832&  0.836 &  0.838&  0.844&  0.845&  0.848&  0.841&  0.803 \\
               & $F$      &  0.822&  0.840&  0.851&  0.855 &  0.855&  0.866&  0.865&  0.861&  0.856&  0.821 \\
\hline $n=100$ & $\chi^2$ &  0.890&  0.905&  0.913&  0.917 &  0.918&  {\bf 0.928}&  {\bf 0.929}&  {\bf 0.928}&  0.916&  0.887 \\
               & $F$      &  0.895&  0.917&  0.920&  0.921 &  0.925&  {\bf 0.934}&  {\bf 0.932}&  {\bf 0.931}&  0.925&  0.897 \\
\hline $n=200$ & $\chi^2$ &  0.911&  0.919&  0.919&  0.925 &  0.926&  0.929&  0.929&  {\bf 0.930}&  0.922&  0.906 \\
               & $F$      &  0.913&  0.920&  0.925&  0.927 &  {\bf 0.930}&  0.929&  {\bf 0.933}&  {\bf 0.933}&  0.926&  0.907 \\

\hline Model 2 & $a$      & -1    & -0.5  & -0.25 & -0.125 & 0     & 0.5   & 0.67  & 1     & 1.5   & 3 \\
\hline $n=50$  & $\chi^2$ &  0.820&  0.839&  0.845&  0.850 &  0.856&  0.863&  0.861&  0.854&  0.838&  0.791 \\
               & $F$      &  0.832&  0.852&  0.864&  0.869 &  0.868&  0.873&  0.872&  0.867&  0.858&  0.804 \\
\hline $n=100$ & $\chi^2$ &  0.861&  0.879&  0.886&  0.889 &  0.890&  0.898&  0.900&  0.901&  0.892&  0.861 \\
               & $F$      &  0.870&  0.889&  0.893&  0.896 &  0.900&  0.903&  0.905&  0.908&  0.901&  0.869 \\
\hline $n=200$ & $\chi^2$ &  0.918&  {\bf 0.932}&  {\bf 0.939}&  {\bf 0.940} &  {\bf 0.942}&  {\bf 0.943}&  {\bf 0.939}&  {\bf 0.943}&  {\bf 0.940}&  0.916 \\
               & $F$      &  0.922&  {\bf 0.934}&  {\bf 0.942}&  {\bf 0.942} &  {\bf 0.944}&  {\bf 0.946}&  {\bf 0.945}&  {\bf 0.945}&  {\bf 0.942}&  0.919 \\

\hline Model 3 & $a$      & -1    & -0.5  & -0.25 & -0.125 & 0     & 0.5   & 0.67  & 1     & 1.5   & 3 \\
\hline $n=50$  & $\chi^2$ &  0.801&  0.817&  0.823&  0.826 &  0.828&  0.833&  0.831&  0.831&  0.823&  0.782 \\
               & $F$      &  0.814&  0.833&  0.840&  0.843 &  0.844&  0.849&  0.848&  0.848&  0.845&  0.793 \\
\hline $n=100$ & $\chi^2$ &  0.879&  0.898&  0.901&  0.903 &  0.906&  0.911&  0.910&  0.911&  0.901&  0.867 \\
               & $F$      &  0.888&  0.902&  0.907&  0.911 &  0.913&  0.916&  0.918&  0.921&  0.909&  0.875 \\
\hline $n=200$ & $\chi^2$ &  0.914&  0.924&  0.927&  0.928 &  0.929&  0.929&  {\bf 0.935}&  {\bf 0.935}&  {\bf 0.936}&  0.908 \\
               & $F$      &  0.915&  0.925&  0.928&  {\bf 0.932} &  {\bf 0.933}&  {\bf 0.931}&  {\bf 0.937}&  {\bf 0.937}&  {\bf 0.937}&  0.908 \\

\hline Model 4 & $a$      & -1    & -0.5  & -0.25 & -0.125 & 0     & 0.5   & 0.67  & 1     & 1.5   & 3 \\
\hline $n=50$  & $\chi^2$ &  0.797&  0.821&  0.829&  0.832 &  0.832&  0.847&  0.848&  0.852&  0.846&  0.794 \\
               & $F$      &  0.812&  0.839&  0.846&  0.850 &  0.854&  0.862&  0.864&  0.865&  0.863&  0.808 \\
\hline $n=100$ & $\chi^2$ &  0.864&  0.882&  0.891&  0.895 &  0.896&  0.903&  0.904&  0.905&  0.902&  0.891 \\
               & $F$      &  0.870&  0.890&  0.899&  0.899 &  0.900&  0.909&  0.912&  0.912&  0.910&  0.898 \\
\hline $n=200$ & $\chi^2$ &  0.914&  0.921&  0.927&  {\bf 0.930} &  {\bf 0.934}&  {\bf 0.938}&  {\bf 0.937}&  {\bf 0.940}&  {\bf 0.944}& 0.928 \\
               & $F$      &  0.917&  0.926&  {\bf 0.931}&  {\bf 0.935} &  {\bf 0.937}&  {\bf 0.942}&  {\bf 0.943}&  {\bf 0.944}&  {\bf 0.947}&  {\bf 0.930} \\
\hline
\end{tabular}
\end{center}
\caption{Values of the test statistics for different parameters $a$ for $1-\alpha =0.95$.}
\label{tabla2}
\end{table}

\section{Conclusions}

In this paper we have introduced a new family of empirical phi-divergence test statistics for testing the simple hypothesis for the logistic regression model when empirical likelihood is considered. Empirical likelihood has revealed as a powerful tool for dealing with many different problems, as it combines the flexibility of non-parametric procedures with the results of parametric inference. The classical way to treat this problem relays on a procedure based on differences between the likelihood through the null hypothesis and the likelihood just considering that a logistic regression model holds; on the other hand, maximum likelihood is a special case of a broad class of functions measuring differences between two probability distributions, the family of phi-divergence measures. We have applied this fact to define the family of empirical phi-divergence test statistics. Next, we have shown that all members of test statistics in this family share the same asymptotic distribution under the null hypothesis, and that this distribution is the same as the classical ELRT. Therefore, if the null hypothesis holds, all these empirical phi-divergence test statistics are equivalent when the sample size is large enough and they can only differ for small and moderate sample sizes.

We have also obtained some results about the power function. In this case, we have found that the power function depends on the divergence measure considered.

To compare the different empirical power divergence test statistics in this case, we have carried out a simulation study. We have considered four different models and study the behavior of the empirical power divergence test statistics. As a conclusion, we have found that there are some empirical power divergence test statistics that can compete with the classical ELRT.

As explained in the paper, when the null hypothesis is simple, the empirical $\phi $-divergence test statistics do not rely on estimations and only a term depending on the sample and $\bm \beta_0$ is taken into account. However, this is not the case if a composite null hypothesis is considered; in this case, estimations of the parameters are needed and the standard procedure considers EMLE. On the other hand, it is possible to generalize EMLE via phi-divergence measures. Thus, the test statistic in this situation could be extended not only from the point of view of extending the test statistics, but also extending the way of estimating the parameters via a procedure based on divergence measures. This approach has been considered in other fields \cite{femipa15, bamapa15, bamapa17} and we aim to treat it in a future research for the case of logistic regression and empirical likelihood.

\section*{Acknowledgements}

This paper has been supported by the Ministry of Economy and Competitiveness of Spain under Grant PGC2018-095194-B-100 and by the Interdisciplinary Mathematical Institute of Complutense University.
%

\bibliographystyle{plain}


\begin{thebibliography}{10}

\bibitem{bamapa15}
N.~Balakrishnan, N.~Martin, and L.~Pardo.
\newblock Empirical phi-divergence test statistics for testing simple and
  composite hypotheses.
\newblock {\em Statistics}, 49(5):951--977, 2015.

\bibitem{bamapa17}
N.~Balakrishnan, N.~Martin, and L.~Pardo.
\newblock Empirical phi-divergence test statistics for the difference of means
  of two populations.
\newblock {\em Advances in Statistical Analysis}, 101(2):199--226, 2017.

\bibitem{bha43}
A.~Bhattacharyya.
\newblock On a measure of divergence between two statistical populations
  defined by their probability distributions.
\newblock {\em Bulletin of the Calcutta Mathematical Society}, 35:99--109,
  1943.

\bibitem{biyo96}
A.M. Bianco and V.J. Yohai.
\newblock Robust estimation in the logistic regression model.
\newblock In H.~Rieder, editor, {\em Robust Statistics, Data analysis and
  Computer Intensive Methods}, volume 109 of {\em Lecture Notes in Statistics},
  pages 17--34. Springer, New York (USA), 1996.


\bibitem{crpa02}
N.~Cressie and L.~Pardo.
\newblock Phi-divergence statistics.
\newblock In A.H. Elshaarawi and W.W. Plegorich, editors, {\em Encyclopedia of
  environmetrics}, volume~13, 1551--1555. John Wiley and Sons, New York
  (USA), 2002.

\bibitem{crre84}
N.~Cressie and T.R.C. Read.
\newblock Multinomial goodness-of-fit tests.
\newblock {\em Journal of the Royal Statistical Society Series B}, 8:440--464,
  1984.


\bibitem{dal86}
J.R. Dale.
\newblock Asymptotic normality of goodness-of-fit statistics for sparse product
  multinomials.
\newblock {\em Journal of the Royal Statistical Society Series B}, 41:48--59,
  1986.

\bibitem{femimapa18}
A.~Felipe, P.~Miranda, N.~Mart\'{\i}n, and L.~Pardo.
\newblock Testing with exponentially tilted empirical likelihood.
\newblock {\em Methodology and Computing in Applied Probability}, 20: 1319--1358, 2018.

\bibitem{femipa15}
A.~Felipe, P.~Miranda, and L.~Pardo.
\newblock Minimum $\phi $-divergence estimation in constrained latent class
  models for binary data.
\newblock {\em Psychometrika}, 80(4):1020--1042, 2015.

\bibitem{fra57}
R.A. Fraser.
\newblock {\em Nonparametric methods in Statistics}.
\newblock John Wiley and Sons, New York (USA), 1957.

%

\bibitem{memopasa95}
M.L. Men\'{e}ndez, D.~Morales, L.~Pardo, and M.~Salicr\'{u}.
\newblock Asymptotic behavior and statistical applications of divergence
  measures in multinomial populations: {A} unified study.
\newblock {\em Statistical Papers}, 36:1--29, 1995.

\bibitem{mepapa97}
M.L. Men\'{e}ndez, L.~Pardo, and M.C. Pardo.
\newblock Asymptotic approximations for the distributions of the $(h,\phi
  )$-divergence goodness-of-fit statistics: {A}pplications to {R}\'{e}nyi's
  statistic.
\newblock {\em Kybernetes}, 26:442--452, 1997.

\bibitem{owe88}
A.B. Owen.
\newblock Empirical likelihood ratio confidence interval for a single
  functional.
\newblock {\em Biometrika}, 75:308--313, 1988.

\bibitem{owe90}
A.B. Owen.
\newblock Empirical likelihood confidence regions.
\newblock {\em The Annals of Statistics}, 18:90--120, 1990.

\bibitem{owe91}
A.B. Owen.
\newblock Empirical likelihood for linear models.
\newblock {\em The Annals of Statistics}, 19:1725--1747, 1990.

\bibitem{owe03}
A.B. Owen.
\newblock {\em Empirical likelihood}.
\newblock Chapman and Hall CRC, Boca Raton (USA), 2001.

\bibitem{par06}
L.~Pardo.
\newblock {\em Statistical inference based on divergence measures}.
\newblock Chapman and Hall CRC, Boca Raton (USA), 2006.

\bibitem{qila94}
J.~Qin and J.~Lawless.
\newblock Empirical likelihood and general estimating equations.
\newblock {\em Annals of Statistics}, 22(1):300--325, 1994.

\bibitem{ren61}
A.~R\'{e}nyi.
\newblock On measures of entropy and information.
\newblock In {\em Proceedings of the Fourth Berkeley Symposium on Mathematical
  Statistics and Probability}, pages 547--561, 1961.

\bibitem{roc01}
M.~La Rocca.
\newblock Robust inference in the logist regression model.
\newblock In M.~Vichi M.~Schader S.~Borra, R.~Rocci and H.~Rieder, editors,
  {\em Advances in Classification and Data Analysis}, volume 109, pages
  209--216. Springer, New York (USA), 2001.

\bibitem{roc03}
M.~La Rocca.
\newblock Bootstrap calibration and empirical likelihood in the logistic
  regression model.
\newblock {\em Journal of the Japanese Society of Computational Statistics},
  15(2):247--254, 2003.

\bibitem{shmi77}
B.D. Sharma and D.P. Mittal.
\newblock New non-additive measures of relative information.
\newblock {\em Journal of Combinatorics, Information and Systems Science},
  2:122--133, 1977.

\end{thebibliography}

\end{document}